\documentclass[11pt, oneside]{amsart}
\usepackage[margin=3cm]{geometry}

\usepackage[utf8]{inputenc}
\usepackage{amsfonts}
\usepackage{amsmath}
\usepackage{amssymb}
\usepackage{amsthm}
 \usepackage{float} 
 \usepackage{graphicx}
\usepackage{soul}
\usepackage{caption}
\usepackage{subcaption}
\usepackage{pinlabel}
\usepackage{tikz}
\usepackage{color}
\usepackage{mathtools}
\usetikzlibrary{shapes.gates.logic.US,trees,positioning,arrows}
\usepackage{thm-restate}

\usepackage{soul}
\usepackage{graphicx}
\usepackage{tikz}
\usepackage[colorlinks=false]{hyperref}
\usepackage{enumerate}
 \usepackage{float} 
 \usepackage{caption}
\usepackage{color}
\usepackage{soul}

\newtheorem{thm}{Theorem}[section]
\newtheorem{prop}[thm]{Proposition}
\newtheorem{lem}[thm]{Lemma}
\newtheorem{cor}[thm]{Corollary}

\theoremstyle{definition}
\newtheorem{defn}[thm]{Definition}
\newtheorem{rem}[thm]{Remark}

\newtheorem{exmp}[thm]{Example}
\newtheorem{ques}[thm]{Question}

\newcommand{\field}[1]{\mathbb{#1}}
\newcommand{\Z}{\field{Z}}

\newcommand{\R}{\field{R}}

\newcommand{\PP}{\field{P}}

\newcommand{\mc}{\mathcal}

\DeclareMathOperator{\Cay}{Cay}

\DeclareMathOperator{\Lk}{Lk}

\DeclareMathOperator{\Stab}{Stab}

\DeclareMathOperator{\proj}{\operatorname{proj}}

\DeclareMathOperator{\diam}{diam}

\usepackage{ifthen}

\newcommand{\showcomments}{yes}

\newsavebox{\commentbox}
{\ifthenelse{\equal{\showcomments}{yes}}%
{\footnotemark
    \begin{lrbox}{\commentbox}
    \begin{minipage}[t]{1.25in}\raggedright\sffamily\upshape\tiny
    \footnotemark[\arabic{footnote}]}%
{\begin{lrbox}{\commentbox}}}%
{\ifthenelse{\equal{\showcomments}{yes}}%
{\end{minipage}\end{lrbox}\marginpar{\usebox{\commentbox}}}%
{\end{lrbox}}}

\title{Largest hyperbolic actions of  3--manifold groups}

\author{Carolyn Abbott}
\address{Department of Mathematics\\
Brandeis University\\
415 South Street\\
Waltham, MA 02453}
\email{carolynabbott@brandeis.edu}

\author{Hoang Thanh Nguyen}
\address{Department of Mathematics\\
University of Danang-University of Science and Education\\
DaNang, Vietnam}
\email{nthoang.math@gmail.com}

\author{Alexander J. Rasmussen}
\address{Department of Mathematics\\
University of Utah\\
Salt Lake City, UT 84112\\
USA}
\email{rasmussen@math.utah.edu}
\date{}

\begin{document}

\date{\today}

\begin{abstract}
The set of equivalence classes of cobounded actions of a group $G$ on different hyperbolic metric spaces carries a natural partial order. Following Abbott--Balasubramanya--Osin, the group $G$ is \textit{$\mathcal{H}$--accessible} if the resulting poset has a largest element. In this paper, we prove that every non-geometric 3-manifold has a finite cover with $\mathcal{H}$--inaccessible fundamental group and give conditions under which the fundamental group of the original manifold is $\mc H$--inaccessible. 
We also prove that every Croke-Kleiner admissible group (a class of graphs of groups that generalizes fundamental groups of 3–dimensional
graph manifolds) has a finite index subgroup that is $\mathcal{H}$--inaccessible.
\end{abstract}



\maketitle

\section{Introduction}
A fixed group $G$  will admit many different
cobounded actions on different hyperbolic metric spaces. Abbott, Balasubramanya, and Osin in \cite{ABO19} show that the set of equivalence classes of cobounded hyperbolic actions of a group $G$ carries a natural partial order; see Section~\ref{sec:prelim}. The resulting poset is called the {\it poset of
hyperbolic structures} on $G$, denoted $\mathcal{H}(G)$. Roughly speaking, one action is larger than another if the smaller space
can be formed by equivariantly collapsing some subspaces of the larger. The motivation is that the larger an action is in this partial order, the more information about the geometry of the group it should provide.

The posets $\mathcal H(G)$ remain mysterious, especially for groups with features of non-positive curvature, which tend to have uncountable posets of hyperbolic structures \cite{ABO19}. First steps towards understanding these posets were made in \cite{Bal20}, \cite{AR19}, and \cite{AR19b}. While \cite{Bal20} and \cite{AR19b} give  complete descriptions of $\mathcal H(G)$ for the groups in question, it appears essentially impossible to do this, given current technology, for groups with strong features of non-positive curvature. Nonetheless, one aspect of the poset that can be understood in many instances is the (non-)existence of a \textit{largest} element of $\mc H(G)$, that is, an  element that is greater than or equal to every other element of the poset.  When a largest element exists, we say that the group $G$ is \textit{$\mc H$--accessible}; otherwise, the group is \textit{$\mc H$--inaccessible}. 
 The first and third authors show that  many groups with features of non-positive curvature are $\mc H$--inaccessible \cite{AR19b}. 
 In this paper, we follow this direction by extending these results to a  large class of 3-manifold groups.

We first consider fundamental groups of non-geometric 3--manifolds with empty or toroidal boundary.  
If $M$ is a compact, orientable, irreducible non-geometric 3-manifold, then there exists a non-empty minimal union $\mc T$ of disjoint essential tori in $M$ such that each connected component of $M\setminus\mc T$ is Seifert fibered or  hyperbolic. This is called the \emph{torus decomposition of $M$}, and the connected components of $M\setminus\mc T$ are called \emph{pieces}.  Our first result shows that if $M$ contains certain types of pieces, then $\pi_1(M)$ is $\mc H$--inaccessible.  To describe these pieces, we introduce the class of \emph{non-elementary} Seifert fibered manifolds, which are those whose base orbifolds are orientable and hyperbolic.  Non-elementary graph manifolds are those whose Seifert pieces are non-elementary, and a mixed manifold is non-elementary if all of its Seifert fibered pieces and maximal graph manifold components are non-elementary.  Any graph  manifold or Seifert fibered manifold is finitely covered by a non-elementary one.  See Section~\ref{sec:NPC} for a more in-depth discussion.

\begin{restatable}{thm}{NPCmfld}
\label{thm:NPCmld}
Let $M$ be a non-geometric 3-manifold with empty or toroidal boundary. If the torus decomposition of $M$ contains any of the following, then $\pi_1(M)$ is $\mc H$--inaccessible: 
    \begin{enumerate}[(1)]
        \item a hyperbolic piece which contains a boundary torus of $M$; 
        \item two hyperbolic pieces glued along a torus;
        \item an isolated non-elementary Seifert fibered piece; or
        \item a non-elementary maximal  graph manifold component.
    \end{enumerate}
\end{restatable}

We note a  straightforward corollary of Theorem~\ref{thm:NPCmld}. 
\begin{cor}
If $M$ is one of the following types of 3-manifolds, then $\pi_1(M)$ is $\mc H$--inaccessible:
    \begin{enumerate}[(a)]
        \item a finite non-trivial connected sum of finite volume cusped hyperbolic 3-manifolds; 
        \item a non-elementary mixed manifold.
    \end{enumerate}
\end{cor}

\begin{proof}
Part (a) follows from Theorem~\ref{thm:NPCmld}(1).  For part (b), observe that a mixed manifold must contain either: two hyperbolic pieces glued along a torus, an isolated non-elementary Seifert fibered piece, or a non-elementary graph manifold component with boundary. In these cases, we apply Theorem \ref{thm:NPCmld} (2), (3), or (4), respectively. 
\end{proof}

The first and third author give a straightforward criterion to show that a group is $\mc H$--inaccessible \cite[Lemma~1.4]{AR19}; see Lemma~\ref{lem:key}.  In this paper, we show in Lemma~\ref{lem:RHG} that if a peripheral subgroup of a relatively hyperbolic group satisfies this criterion, then the group itself also satisfies the criterion. There is a natural relatively hyperbolic structure on $\pi_1(M)$ in which the peripheral subgroups are fundamental groups of certain tori, Seifert fibered pieces, and graph manifolds.  The proof Theorem~\ref{thm:NPCmld} proceeds by showing that, in each case, there is a peripheral subgroup that is $\mc H$--inaccessible because it satisfies Lemma~\ref{lem:key}, and then applying Lemma~\ref{lem:RHG}. The key step is the following theorem.

\begin{thm} \label{thm:graphandSFmflds}
If $M$ is a non-elementary graph manifold or a non-elementary Seifert fibered manifold, then $\pi_1(M)$ is $\mc H$--inaccessible.
\end{thm}

Theorem \ref{thm:graphandSFmflds} generalizes   \cite[Theorem~1.1]{AR19}, in which the first and third authors show that flip graph manifold  groups, a special class of  graph manifold groups, are  $\mathcal{H}$--inaccessible.

The class of non-elementary graph manifold groups belongs to a  larger class of graphs of groups called {\it Croke-Kleiner admissible groups}  that were introduced by Croke and Kleiner in \cite{CK02}. Roughly speaking, Croke-Kleiner admissible groups are modeled on the JSJ structure of  graph manifolds where the (fundamental groups of the) Seifert fibered pieces are  replaced by central extensions $G$ of general hyperbolic groups $H$:
\begin{equation*}\label{centralExtEQ}
1\to Z(G)=\mathbb Z\to G\to H\to 1.    \end{equation*}
In some sense, Croke-Kleiner admissible groups are the simplest interesting groups  constructed algebraically from a finite number of hyperbolic  groups.  The  $\mathcal{H}$--inaccessibility of non-elementary graph manifold groups  follows from the following  result for Croke-Kleiner admissible groups.
\begin{thm}
\label{thm1}
Every Croke-Kleiner admissible group has a  subgroup of  index at most $2$ which is  $\mathcal{H}$--inaccessible.
\end{thm}

A non-geometric 3-manifold $M$ always has a double cover in which all Seifert fibered pieces are non-elementary, and hence passing to a further finite cover if necessary, we obtain a finite cover $M' \to M$ such that $M'$ is either a non-elementary graph manifold or a non-elementary mixed manifold.
Combining Theorems~\ref{thm:NPCmld} and \ref{thm:graphandSFmflds} yields the following corollary.
\begin{cor}
If $M$ is a non-geometric 3-manifold then $\pi_1(M)$ has a finite index subgroup $H$ such that every finite index subgroup $K \le H$ is $\mathcal{H}$--inaccessible.
\end{cor}

So far, we have only discussed non-geometric $3$--manifolds.  In some cases, we can also understand the $\mc H$--accessibility of (finite-index subgroups) of geometric $3$--manifold groups.

\begin{prop}
Every $3$--manifold with Nil or Sol geometry has a finite cover whose fundamental group is $\mc H$--inaccessible.  The fundamental group of  a closed hyperbolic $3$--manifold is $\mc H$--accessible, while the fundamental group of a finite-volume cusped hyperbolic $3$--manifold is $\mc H$--inaccessible.
\end{prop}

\begin{proof}
If a 3--manifold $M$ has the geometry of Sol, then $M$ is a torus bundle over a $1$--dimensional orbifold (an interval with reflection boundary points or
a circle) and thus $M$
has a double cover that is a  torus bundle with Anosov monodromy. The $\mc H$--inaccessibility of the fundamental group of this bundle then follows from work of the first and third authors \cite{AR19}. 

 If the geometry of $M$ is Nil, $M$ is a Seifert fibered 3-manifold, and $M$ 
is finitely covered by a torus bundle $M'$ with unipotent monodromy, and the only possible hyperbolic actions are lineal and elliptic. On the other hand, the abelianization of $\pi_1(M')$ is virtually $\Z^2$, so this yields infinitely many homomorphisms $\Z^2 \to \R$ modulo scaling, and infinitely many inequivalent actions on $\R$ by translations. Since such lineal actions are incomparable by \cite[Theorem~2.3]{ABO19}, $\pi_1(M')$ is $\mathcal H$-inaccessible.

The fundamental group of a closed hyperbolic $3$--manifold is a hyperbolic group, and so is $\mc H$--accessible.  The result for a finite-volume cusped hyperbolic $3$--manifold is Corollary~\ref{cor:hyp}.
\end{proof}

In Section \ref{sec:Boundary}, we consider general finitely generated 3-manifold groups. This includes fundamental groups of reducible 3-manifolds, certain geometric 3-manifolds, and 3-manifolds with non-toroidal boundary. In Proposition \ref{prop:highergenus}, we characterize the  $H$--accessibility of many such 3-manifold groups. In particular, any finitely generated fundamental group of a hyperbolic 3-manifold without rank-1 cusps is $\mathcal H$--accessible.

Many basic questions about posets of hyperbolic structures are still open. Surprisingly, it is still unknown whether $\mathcal{H}$--inaccessibility of a finite index normal subgroup of $G$ passes to $\mathcal{H}$--inaccessibility of the ambient group $G$. In the setting of non-geometric 3-manifolds, the only cases in which we are unable to determine the $\mathcal H$--(in)accessibility of the fundamental group are a manifold all of whose Seifert fibered pieces are elementary  and a graph manifold whose underlying graph contains a single vertex.  We suspect that these manifolds are also
$\mathcal{H}$--inaccessible, but the techniques in this paper do not apply.
We thus ask the following question.
\begin{ques}
Let $M$ be a graph manifold with underlying graph containing only one vertex. Is $\pi_1(M)$ $\mathcal{H}$--inaccessible?
\end{ques}

\subsection{Acknowledgments}
Abbott was partially supported by NSF grants DMS-1803368 and DMS-2106906. Nguyen was partially supported by Project ICRTM04\_2021.07 of the International Centre for Research and Postgraduate Training in
Mathematics, VietNam. Rasmussen was partially supported by NSF grants DMS-1840190 and DMS-2202986.

\section{Preliminaries}
\label{sec:pre}

\subsection{$\mathcal{H}$--accessibility}
\label{sec:prelim}

 In this section, we review the partial order on cobounded group actions introduced in \cite{ABO19}. 
 Fix a group $G$. If $G$ acts coboundedly on two  metric spaces $X$ and $Y$, we say $G\curvearrowright X$ \textit{is dominated by} $G\curvearrowright Y$, written $G\curvearrowright X \preceq G\curvearrowright Y$, if there exists a coarsely $G$--equivariant coarsely Lipschitz map $Y\to X$.  The preorder $\preceq$ induces an equivalence relation $G\curvearrowright X\sim G\curvearrowright Y$  if and only if $G\curvearrowright X \preceq G\curvearrowright Y$ and $G\curvearrowright Y \preceq G\curvearrowright X$. It descends to a partial order $\preccurlyeq$ on the set of equivalence classes. We denote the equivalence class of an action by $[G\curvearrowright X]$.

\begin{defn}
Given a group $G$, the \emph{poset of hyperbolic structures on $G$} is defined to be \[\mathcal H(G)\vcentcolon = \{[G\curvearrowright X]\mid G \curvearrowright X \textrm{ is cobounded and } X \textrm{ is hyperbolic}\},\] equipped with the partial order $\preccurlyeq$.
\end{defn}

By \cite[Proposition~3.12]{ABO19}, this is equivalent to the original definition of $\mathcal H (G)$ in terms of generating sets. We say that an element of a poset is \emph{largest} when it is greater than or equal to every other element of the poset. Such an element is unique. 

\begin{defn}
A group $G$ is {\it $\mathcal{H}$--accessible} if the poset $\mathcal{H}(G)$ has a largest element. Otherwise, it is  {\it $\mathcal{H}$--inaccesible}.
\end{defn}

The following lemma gives a simple criterion to check if a group is  $\mathcal{H}$--inaccessible.

\begin{lem}[{\cite[Lemma~1.4]{AR19}}]
\label{lem:key}
Let $G$ be a group. Suppose that there are commuting  elements $a, b \in G$  and hyperbolic actions $G \curvearrowright X$ and $G \curvearrowright Y$ such that
\begin{enumerate}
    \item $a$ acts loxodromically and $b$ acts elliptically in the action $G \curvearrowright X$; and 
    \item $b$ acts loxodromically in the action $G \curvearrowright Y$.
\end{enumerate}
Then there does not exist a hyperbolic action $G\curvearrowright Z$ such that $G \curvearrowright X \preceq G \curvearrowright Z$ and $G \curvearrowright Y \preceq G \curvearrowright Z$.
\end{lem}

We will typically apply this lemma to spaces $X$ and $Y$ that are quasi-isometric to lines. 
 These actions on quasi-lines will be constructed using quasimorphisms, which, in turn, will be constructed using hyperbolic actions of $G$ with very weak proper discontinuity properties (see \cite{BBF15}). 

\begin{defn}
Let $G\curvearrowright X$ be a hyperbolic action and $g\in G$ be loxodromic with fixed points $\{g^{\pm}\}\subset \partial X$. The element $g$ is \emph{WWPD} if the orbit of the pair $(g^+,g^-)$ is discrete in the space $\partial X \times \partial X \setminus \Delta$, where $\Delta=\{(x,x) \mid x\in \partial X\}$ is the diagonal subset. A WWPD element $g$ is \emph{WWPD${}^+$} if any $h\in G$  that stabilizes   $\{g^{\pm}\}$ as a \emph{set} also fixes $\{g^{\pm}\}$ pointwise.
\end{defn}

The following proposition follows from \cite[Corollary 3.2]{BBF16}.

\begin{prop}
\label{prop1}
Suppose that $G\curvearrowright X$ is a hyperbolic action with a WWPD${}^+$ element $g\in G$. There is a homogeneous quasimorphism $q \colon G \to \mathbb R$  such that $q(g) \neq 0$ and $q(h) =0$ for any element $h \in G$ that acts elliptically on $X$.
\end{prop}

Homogeneous quasimorphisms, in turn, give rise to actions on quasi-lines.

\begin{prop}[{\cite[Lemma~4.15]{ABO19}}]
\label{prop2}
Let $q \colon G \to \mathbb{R}$ be a non-zero homogeneous quasimorphism. Then there is an action of $G$ on a quasi-line $\mathcal{L}$ such that $q(g) \neq 0$ if and only if $g$ acts loxodromically on $\mathcal{L
}$.
\end{prop}

\subsection{Projection axioms}
\label{sec:projaxioms}
We will use the Bestvina--Bromberg--Fujiwara projection complex machinery developed in \cite{BBF15} to obtain
actions on quasi-trees. In this section, we review this machinery.

\begin{defn}
\label{defn:Projectionaxioms}

Let $\mathbb{Y}$ be a collection of geodesic spaces equipped with projection maps  $$\{\pi_{Y}\colon  \mathbb Y\setminus \{Y\}\to Y\}_{Y\in \mathbb Y}.$$  Let $d_{Y}(X,Z) = \operatorname{diam} (\pi_{Y}(X) \cup \pi_{Y}(Z))$ for $X\ne Y\ne Z \in \mathbb{Y}$.  The pair $(\mathbb{Y}, \{\pi_{Y}\}_{Y\in\mathbb Y})$ satisfies the {\it projection axioms} for a {\it projection constant} $\xi \ge 0$ if
 \begin{enumerate}
     \item 
     \label{axiom1}  $\operatorname{diam} (\pi_{Y}(X)) \le \xi$ whenever $X \neq Y$;
     \item 
     \label{axiom2} if $X,Y,Z$ are distinct and $d_{Y}(X, Z) > \xi$, then $d_{X}(Y,Z) \le \xi$; and
     \item 
     \label{axiom3} for $X \neq Z$, the set $\{ Y \in \mathbb{Y} \,\mid \, d_{Y}(X, Z) > \xi \}$ is finite.
\end{enumerate}
\end{defn}

For a fixed $K>0$ and a pair $(\mathbb{Y}, \pi_{Y})$ satisfying the projection axioms for some constant $\xi$,  Bestvina, Bromberg, and Fujiwara construct a {\it quasi-tree of spaces} $\mathcal C_K(\mathbb Y) $ in \cite{BBF15}.  
  If $\mathbb Y$ admits an action of the group $G$ so that $\pi_{gY} (gX)=g\pi_Y(X)$ for any $g\in G$ and $X,Y\in \mathbb Y$, then $G$ acts by isometries on $\mathcal C_K(\mathbb Y)$; see Section~4.4 in \cite{BBF15}.
  Moreover, they show that if $K>4\xi$ and $\mathbb Y$ is a collection of metric spaces uniformly quasi-isometric to $\mathbb R$, then $\mathcal C_K(\mathbb Y)$ is an unbounded quasi-tree \cite[Theorem~4.14]{BBF15}.

The following is a useful example to keep in mind throughout the paper and is discussed in detail in the introduction of \cite{BBF15}.

\begin{exmp}\label{ex:linesinhypplane}
Let $G$ be a discrete group of isometries of $\mathbb H^2$ and $g_1,\dots, g_k\in G$ a finite collection of loxodromic elements with axes $\gamma_1,\dots, \gamma_k$, respectively.  Let $\mathbb Y$ be the set of all $G$--translates of $\gamma_1,\dots, \gamma_k$, and given $Y\in \mathbb Y$, define $\pi_Y$ to be closest point projection in $\mathbb H^2$.  Since all translates of each $\gamma_i$ are convex, this is a well-defined $1$--Lipschitz map, and it follows from hyperbolicity that the projection of one translate of an axis onto another has uniformly bounded diameter.  In \cite{BBF15} it is verified that the pair $(\mathbb Y,\pi_Y)$ satisfies the projection axioms for some constant $\xi$.
\end{exmp}

Given a pair $(\mathbb Y,\{\pi_Y\}_{Y\in\mathbb Y})$ satisfying the projection axioms and three domains $X,Y,Z\in \mathbb Y$, there is a notion of distance between a point of $X$ and a point of $Z$ from the point of view of $Y$, which we now describe. Let $x\in X$ and $z\in Z$.   If $X,Y,Z$ are all distinct,  then define 
$d_Y(x,z)\vcentcolon=d_Y(X,Z)$.  If $Y=X$ and $Y\ne Z$, then define $d_Y(x,z)\vcentcolon=\operatorname{diam}(\{x\}\cup\pi_Y(z))$, where the diameter is measured in $Y$.  Finally, if $X=Y=Z$, then let $d_Y(x,z)$ be the distance in $Y$ between $x$ and $z$.  
The spaces  $Y\in \mathbb Y$ naturally embed into $\mathcal C_K(\mathbb Y)$, so the distance $d_{\mathcal C_K(\mathbb Y)}(x,z)$ is also defined.
 
We have the following upper bound on distance in $\mathcal C_K(\mathbb Y)$. Set $[t]_K=t$ if $t\ge K$ and $[t]_K=0$ if $t<K$.

\begin{prop}[{\cite[Lemma~4.4]{BBF15}}]
 \label{prop:BBFDistanceProp}
 Let $(\mathbb Y, \pi_{Y\in \mathbb Y})$ satisfy the projection axioms with constant $\xi$. For any $K$ sufficiently large,
$$
d_{\mathcal C_K(\mathbb Y)}(x, z) \leq 6K + 4 \sum_{Y \in \mathbb Y} [d_Y(x,z)]_K.
$$ 
\end{prop}

This distance formula was originally stated with a modified distance function.  In \cite{BBF15} the distance defined above was denoted $d_Y^\pi$, and the modified distance was denoted $d_Y$. 
  However, since $d_Y(x,z)\leq d_Y^\pi(x,z)$ for all point $x,z$, the inequality holds with the distance $d_Y^\pi$, as well.  As we will not need the modified distance function in this paper, we use the simpler notation $d_Y$ for $d_Y^\pi$ and choose to state the proposition with this distance.

\subsection{Croke-Kleiner admissible groups}
In this section, we  review Croke-Kleiner admissible groups \cite{CK02} and the associated Bass-Serre space, a notion defined in \cite{HRSS22}.

\begin{defn}
A \textit{graph of groups} $\mathcal{G} = (\Gamma, \{G_{\mu}\}, \{G_{\alpha} \}, \{\tau_{\alpha} \})$ is a connected graph $\Gamma$ together with
a group ${G}_{\sigma}$ for each $\sigma \in  V(\Gamma) \cup E(\Gamma)$ (here $V(\Gamma)$ and $E(\Gamma)$ denote vertices and edges), and an injective homomorphism $\tau_{\alpha} \colon {G}_{\alpha} \to {G}_{\mu}$ for each oriented edge $\alpha$, where $\mu$ denotes the terminal vertex of $\alpha$.
\end{defn}

\begin{defn}
\label{defn:admissible}
A graph of groups $\mathcal{G} = (\Gamma, \{G_{\mu}\}, \{G_{\alpha} \}, \{\tau_{\alpha} \})$  is called {\it {admissible}}  if the following hold.
\begin{enumerate}
    \item $\mathcal{G}$ is a finite graph with at least one edge.
    \item Each vertex group ${ G}_{\mu}$ has center $Z({ G}_{\mu}) \cong \Z$, ${ H}_{\mu} \vcentcolon= { G}_{\mu} / Z({ G}_{\mu})$ is a non-elementary hyperbolic group, and every edge group ${ G}_{\alpha}$ is isomorphic to $\Z^2$.
    \item Let $\alpha_1$ and $\alpha_2$ be distinct edges oriented towards a vertex $\mu$, and for $i = 1,2$ let $K_i \subset { G}_{\mu}$ be the image of the edge homomorphism ${G}_{\alpha_i} \to {G}_{\mu}$. Then for every $g \in { G}_{\mu}$, $gK_{1}g^{-1}$ is not commensurable with $K_2$, and for every $g \in  G_{\mu} \setminus K_i$, $gK_ig^{-1}$ is not commensurable with $K_i$. 
    \item For every edge group ${ G}_{\alpha}$ with $\alpha = [\alpha^-, \alpha^+]$ (oriented from $\alpha^-$ to $\alpha^+$), the subgroup of $G_{\alpha}$ generated by $\tau_{\alpha}^{-1}(Z({ G}_{\alpha^+}))$ and $\tau_{\overline{\alpha}}^{-1}(Z({ G}_{\alpha^-}))$ has finite index in ${ G}_{\alpha}$.
\end{enumerate}
A group $G$ is \emph{admissible} if it is the fundamental group of an admissible graph of groups.  Such groups are often called \textit{Croke-Kleiner admissible groups}.
\end{defn}

\begin{lem}[{\cite[Lemma 4.2]{HRSS22}}]
\label{lem:HRSS22}
    Let $\mathcal{G} = (\Gamma, \{G_{\mu}\}, \{G_{\alpha} \}, \{\tau_{\alpha} \})$ be  a Croke-Kleiner admissible group. For each edge $\alpha=[\alpha^-,\alpha^+]$ of $\mathcal{G}$, denote
    \[
    C_{\alpha} = \tau_{\alpha} (\tau^{-1}_{\bar{\alpha}}(Z_{\alpha^{-}})),
    \] which is a subgroup of $G_{\alpha^+}$. Each vertex group $G_{\mu}$ has an infinite generating set $S_{\mu}$ so that the following holds.
    \begin{enumerate}
        \item $\Cay(G_{\mu}, S_{\mu})$ is quasi-isometric to a line.
        \item The inclusion map $Z_{\mu} \to \Cay(G_{\mu}, S_{\mu})$ is a $Z_{\mu}$--equivariant quasi-isometry.
        \item For each edge $\alpha$ with $\alpha^{+}= \mu$ we have that $C_{\alpha}$ is uniformly bounded in $\Cay(G_{\mu}, S_{\mu})$.
    \end{enumerate}
\end{lem}

\begin{rem}
\label{rem:loxellipaction}
    The quasi-line $\Cay(G_{\mu}, S_{\mu})$ satisfies:
    \begin{itemize}
        \item The center $Z_{\mu}$ of $G_{\mu}$ acts loxodromically on $\Cay(G_{\mu}, S_{\mu})$.
        \item If $\omega$ is an adjacent vertex to $\mu$ in $\Gamma$, then each cyclic subgroup of $G_\mu$ conjugate to $Z_{\omega}$ acts elliptically on $\Cay(G_{\mu}, S_{\mu})$.
    \end{itemize}
\end{rem}

Let $\mathcal{G}$ be a graph of finitely generated groups, and let $G \curvearrowright T$ be the action of $G=\pi_1(\mathcal G)$ on the associated Bass-Serre tree {{of $\mathcal{G}$}} ({{we refer the reader to Section~2.5 in \cite{CK02} for a brief discussion}}). For each vertex $v$ of the Bass-Serre tree $T$, let $\check{v}$ denote the vertex $\mu$ of $\Gamma$ so that $v$ represents $gG_{\mu}$ for some $g$ in $G$.
For each vertex group $G_\mu$ and edge group $G_\alpha$, fix once and for all finite symmetric generating sets $J_\mu$ and $J_\alpha$ respectively, such that $J_\alpha=J_{\bar{\alpha}}$ and $\tau_\alpha\left(J_\alpha\right) \subseteq J_{\alpha^{+}}$.

We briefly sketch the description of the Bass-Serre space $X$ for the graph of groups $\mathcal G$ and refer the reader to \cite{HRSS22} for a full description of the space.  Given a vertex $v$ of $T$, the associated vertex space $X_v$ of $X$ is a graph isometric to $\Cay(G_{\check{v}},J_{\check{v}})$.   If $e$ is a (directed) edge in $T$, then the associated edge space $X_e$ is isometric to $\Cay(G_{\check{e}},J_{\check{e}} )$.  
Edges are added between the vertex and edge spaces so that the maps $\tau_{\check{e}}$ induce isometric embeddings of the edge spaces into the vertex spaces, which we denote by $\tau_e\colon X_e\to X_{e^+}$ and $\tau_{\bar e}\colon X_e\to X_{e^-}$.  

Suppose $\mathcal G$ is an admissible graph of groups with Bass-Serre tree $T$ and Bass-Serre space $X$. For each vertex $\mu$ of $\Gamma$, let $S_\mu$ be given by Lemma~\ref{lem:HRSS22}. Without loss of generality, we can assume that $J_{\mu}$ is contained in $S_{\mu}$, where $J_{\mu}$ is the fixed generating set of $G_{\mu}$.

\begin{defn}[Subspaces $L_v$ and $\mathcal H_v$]
\label{defn:quasilineL_v}
 Suppose the vertex $v\in T$ represents $gG_{\check{v}}$. Let $L_v$  be the graph with vertex set $gG_{\check{v}}$ and with an edge connecting $x, y\in gG_{\check{v}}$ if $x^{-1}y\in S_{\check{v}}$.  In particular, $L_v$ is isometric to $\Cay(G_{\check{e}}, S_{\check{e}})$, which is a quasi-line by Lemma~\ref{lem:HRSS22}.
 
 Let $\mathcal H_v$  be the graph with vertex set $gG_{\check{v}}$ and with an edge connecting $x,y\in gG_{\check{v}}$ if $x^{-1}y\in J_{\check{v}}\cup Z_{\check{v}}$. It is isometric to $\Cay(G_{\check{v}},J_{\check{v}}\cup Z_{\check{v}})$.
\end{defn}

Since $L_v$ and $\mathcal H_v$ are each obtained from $X_v$ by adding extra edges, there are distance non-increasing maps $p_v \colon X_v \to L_v$ and $i_v\colon X_v\to \mathcal H_v$ that are the identity on vertices.  The space $\mathcal H_v$ is constructed to represent the geometry of $H_{\check{v}} = G_{\check{v}}/Z_{\check{v}}$ and is relatively hyperbolic:

\begin{lem}[{\cite[Lemma~2.15]{HRSS22}}]
\label{lem:RHGHRSS22}
   $\mathcal{H}_v$ is hyperbolic relative to the collection 
\[
\mathcal{P}_v = 
\{ \ell_{e} \vcentcolon = i_{v} (\tau_{e}(X_e)) \, |\, e \in E(T)\,\,\text{such that}\,\, e^{+} = v \}.
\] 
\end{lem}
It follows from \cite{Sis13} that  there is a  coarse closest point projection map
$$\proj_{\ell_e} \colon \mathcal{H}_v \to \ell_e$$
that is coarsely Lipschitz with constants independent of $e$ and $v$.

\begin{rem}
\label{rem:boundeddiameter}
As peripheral subsets in a relatively hyperbolic space, the sets
$
\{ \ell_{f} \, |\, f \in E(T) \textrm{ and }  f^{+} = v \}
$ together with the maps $\proj_{\ell_f}$ satisfy the projection axioms for a constant $\xi_0$. 
\end{rem}

We now show that if $e$ is an edge from $u$ to $v$,  the various maps defined above can be composed to form a quasi-isometry between the quasi-line $\ell_e\in\mathcal P_u$ and the quasi-line $L_v$.  Let $  \psi_{e} \colon \ell_{\bar{e}} \to L_v$ be the map from \cite{HRSS22} which is defined as the restriction to $\ell_{\bar e}$ of the composition
  \begin{equation}\label{eq:psi_e}
p_v \circ \tau_e \circ \tau_{\bar{e}}^{-1} \circ i_{u}^{-1}.
  \end{equation}   In \cite{HRSS22}, the authors prove that $\psi_e$  is coarsely Lipschitz and note that $\psi_e$ is, in fact, a quasi-isometry. We prove a more general result:

\begin{lem}\label{lem:qlineQI}
Let $\ell_1,\ell_2$ be two quasi-lines, and suppose a group $G$ acts coboundedly on both $\ell_1$ and $\ell_2$.  Any $G$--equivariant coarsely Lipschitz map $\psi\colon \ell_1\to\ell_2$ is a quasi-isometry.
\end{lem}

\begin{proof}
  Since there is a $G$--equivariant coarsely Lipschitz map from $\ell_1$ to $\ell_2$, we have $[G\curvearrowright \ell_1]\succcurlyeq [G\curvearrowright \ell_2]$ in the poset $\mc H(G)$.  However, since  $G\curvearrowright \ell_1$ and $G\curvearrowright \ell_2$ are both lineal, \cite[Theorem~4.22]{ABO19} implies that these actions must be equivalent.  Thus there is a coarsely $G$--equivariant quasi-isometry $\Phi\colon \ell_1\to \ell_2$.  We will show that $\Phi$ and $\psi$ differ by a uniformly bounded amount, which will then show that $\psi$ is also a quasi-isometry.
  
  Fix a basepoint $x_0\in \ell_1$.  Since $G\curvearrowright \ell_1$ is cobounded, there is a constant $B$ such that for any $x\in \ell_1$, there is some $g\in G$ such that $d_{\ell_1}(x,gx_0)\leq B$.  Since $\psi$ is coarsely Lipschitz and $\Phi$ is a quasi-isometry, there is a constant $A$, depending on $B$ and the coarse Lipschitz constants for $\Phi$ and $\psi$, such that $d_{\ell_2}(\Phi(x), \Phi(gx_0))\leq A$ and $d_{\ell_2}(\psi(x),\psi(gx_0))\leq A$.  Moreover, since $\Phi$ is coarsely $G$--equivariant, there is a constant $C$ such that  $d_{\ell_2}(\Phi(gx_0),g\Phi(x_0))\leq C$.  Let $D=d_{\ell_2}(\Phi(x_0),\psi(x_0))$.
  
  By the triangle inequality and $G$--equivariance of $\psi$, we have
  \begin{align*}
  d_{\ell_2}(\Phi(x),\psi(x)) & \leq d_{\ell_2}(\Phi(x),g\Phi(x_0)) + d_{\ell_2}(g\Phi(x_0),g\psi(x_0)) +  d_{\ell_2}(\psi(gx_0),\psi(x)) \\
  & \leq (A+C) + D + A,
  \end{align*}
completing the proof. 
  \end{proof}

We now complete the proof that $\psi_e$ is a quasi-isometry.
\begin{lem}
\label{lem:psiquasi}
There are constants $\lambda\geq 1$ and $c\geq 0$ depending only on $\mc G$ such that the following holds.  For any oriented edge $e$ in the Bass-Serre tree $T$ of $\mathcal{G}$, the map $\psi_{e} \colon \ell_{\bar{e}} \to L_v$ is a $(\lambda, c)$--quasi-isometry.  
\end{lem}

\begin{proof}
In Lemma~6.16 of \cite{HRSS22}, the authors prove that the map $\psi_e$ is coarsely Lipschitz.    Moreover, from the definitions of $\ell_{\bar e}$ and $L_v$ as Cayley graphs with respect to infinite generating sets, $G_{\check e}$ acts by isometries on both, and $\psi_e$ is $G_{\check e}$--equivariant.  Therefore, $\psi_e$ is a quasi-isometry by Lemma~\ref{lem:qlineQI}.  As there are only finitely many $G$--orbits of edges in $T$, we can choose the constants of these quasi-isometries to be independent of the edge $e$. 
\end{proof}


\section{Croke-Kleiner admissible groups and $\mc H$--inaccessibility}
\label{CKA:finiteindex}

In this section, we prove Theorem~\ref{thm1}: every Croke--Kleiner admissible group has a finite index subgroup that is  $\mathcal{H}$--inaccessible.

Fix a Croke--Kleiner admissible group $\mathcal{G} = (\Gamma, \{G_{\mu}\}, \{G_{\alpha} \}, \{\tau_{\alpha} \})$.  We partition the vertex set $T^0$ of the Bass-Serre tree into two disjoint collections of vertices $\mathcal{V}_1$ and $\mathcal{V}_2$ such that if $v$ and $v'$ are in $\mathcal{V}_i$ then $d_{T}(v, v')$ is even. Since any automorphism of $T$ either preserves $\mathcal V_1$ and $\mathcal V_2$ setwise or interchanges them, we have:

\begin{lem}[{\cite[Lemma~4.6]{NY20}}]
\label{lem:index2subgroup}
Let $\mathcal{G} = (\Gamma, \{G_{\mu}\}, \{G_{\alpha} \}, \{\tau_{\alpha} \})$ be a Croke-Kleiner admissible group. 
There exists a subgroup $G' \le G = \pi_1(\mathcal{G})$ of index at most $2$ in $G$ so that $G'$ preserves $\mathcal{V}_1$ and $\mathcal{V}_2$ and $G'$ is also a Croke-Kleiner admissible group.
\end{lem}

Let $G'$ be the finite index subgroup of $G$ given by Lemma~\ref{lem:index2subgroup}.
In light of Lemma~\ref{lem:key}, to show that $G'$ is $\mathcal H$-inaccessible, it suffices to construct commuting $a_i\in G'$ and actions $G'\curvearrowright X_i$ for $i=1,2$ such that $a_i$ is elliptic with respect to the action $G'\curvearrowright X_{3-i}$ and loxodromic with respect to the action $G'\curvearrowright X_i$. Our spaces $X_i$ will be quasi-trees of metric spaces.

\subsection{Construction of group actions}
For notational simplicity, we  replace $G$ by its index $\leq 2$ subgroup $G'$.
For each vertex $v$ in the Bass-Serre tree 
 $T$, let $L_v$ be the quasi-line from Definition~\ref{defn:quasilineL_v}. Recall that $gL_v = L_{gv}$ for any group element $g$ in $G$.

Let $\mathbb{L}_1$ be the collection of quasi-lines $\{ L_v \}_{v \in \mathcal{V}_1}$ and $\mathbb{L}_2$ be the collection of quasi-lines $\{ L_v \}_{v \in \mathcal{V}_2}$.
We define a  projection  of $L_v$ to $L_{v'}$ in $\mathbb L_i$  as follows.

\begin{defn}
[Projection maps in $\mathbb L_i$]
\label{defn:projectionmap}
For any two distinct vertices $v,v'\in \mc V_i$, let  $e'=[w, u]$ and $e=[u,v]$ denote the last two (oriented) edges in $[v',v]$. 
The {\it projection} from $L_{v'}$ into $L_v$ is
\[
\Pi_{L_v} (L_{v'}) \vcentcolon = \psi_{e}(\proj_{\ell_{\bar{e}}}(\ell_{e'})),
\]
where $\psi_e \colon \ell_{\bar{e}} \to L_v$ and $\proj_{\ell_{\bar{e}}} \colon \mathcal{H}_u \to \ell_{\bar{e}}$ are the maps introduced in Section~\ref{sec:pre}.
\end{defn}

\noindent The fact that $d(v, v')$ is even is not necessary for Definition~\ref{defn:projectionmap}, only that $d(v, v') \ge 2$.

We  will  verify that the
$\mathbb{L}_i$   with these projection maps satisfy the projection axioms (see Definition~\ref{defn:Projectionaxioms}) for $i=1,2$.
Let $d_{L_a}(L_b,L_c)$ be the projection distance $\diam \bigl (\Pi_{L_{a}} (L_{b}) \cup \Pi_{L_{a}} (L_{b}) \bigr )$.

\begin{lem}
\label{lem:easy1}
There exists a constant $\lambda >0$ such that $\diam(\Pi_{L_v}(L_{v'}))\leq \lambda$ for any distinct $v,v'\in \mathcal V_i$ for $i=1,2$.  Moreover, if $ {a},  {b}, {c} \in \mathcal V_i$ are distinct vertices with $d_{T}( a, [b, c]) \ge 2$, then $\Pi_{L_a}(L_c) = \Pi_{L_a}(L_b)$.
\end{lem}

\begin{proof}
By Remark~\ref{rem:boundeddiameter}, there is a uniform bound on the diameter of $\proj_{\ell_{\bar{e}}}(\ell_{e'})$. Combined with the fact that $\psi_e$ is uniformly coarsely Lipschitz, this gives the constant $\lambda$. Considering the convex hull of $\{a\}\cup [b,c]$, we see that, orienting $[c,a]$ and $[b,a]$ towards $a$, the last two edges of $[c,a]$ are the same as the last two edges of $[b,a]$. Hence by definition, $ \Pi_{L_a}(L_c) =  \Pi_{L_a}(L_b)$. 
\end{proof}

Let $v$ be a vertex of the Bass-Serre tree $T$. By Remark~\ref{rem:boundeddiameter}, the collection 
$\{ \ell_{f} = i_{v} (\tau_{f}(X_f)) \, |\, f \in E(T)\,\,\text{such that}\, f^{+} = v \}$ satisfies the projection axioms with a constant $\xi_0$.  Let $d_\ell$ denote the projection distances with respect to $\proj_\ell$.  The following lemma follows immediately from Lemma~\ref{lem:psiquasi} and the definitions of $d_{\ell_e}$ and $d_{L_{v}}$.
\begin{lem}
\label{lem:easy3}
There exists a constant $\lambda >0$ such that the following holds. Let $u, v, w$ be distinct vertices in $\mathcal{V}_1$ contained in $\Lk(o)$ for some vertex $o$ in $\mathcal{V}_2$. Let $e=[w,o]$,  $e_1=[u,o],$ and $ e_2=[v,o]$. Then
\[
\frac{1}{\lambda} d_{\ell_{e}}(\ell_{e_1}, \ell_{e_2}) -\lambda \le d_{L_{w}} (L_{u}, L_{v})     \le \lambda d_{\ell_{e}}(\ell_{e_1}, \ell_{e_2}) + \lambda.
\]
\end{lem}

We are now ready  to verify the projection axioms.

\begin{prop}
\label{prop:Fi}
There exists $\xi >0$ such that for each $i \in \{1,2\}$, $\mathbb{L}_i$ together with the projection maps $\proj_\ell$ satisfies the projection axioms.
\end{prop}

\begin{proof}
    We  verify the projection axioms for $\mathbb L_1$. The case for $\mathbb L_2$ is identical. The constant $\xi$ will be defined explicitly during the proof.

    \textbf{Axiom 1:}  This follows from Lemma~\ref{lem:easy1}.

\textbf{Axiom 2:} 
   Let $ {u},  {v}, {w}$ be distinct vertices in $ \mathcal V_1$. In the course of the proof, we will compute a constant $\xi > 0$ such that if $d_{L_w}(L_u, L_v ) > \xi$, then $d_{L_u}(L_w, L_v )  \le \xi $. 
   
   Since $d_{L_w}(L_u, L_v ) > 0$, it follows from Lemma~\ref{lem:easy1} that either $w$ lies on $[u,v]$ or $d_T(w, [u,v]) =1$. If $w$ lies on $[u,v]$, then since $u,w,v\in \mathcal V_1$, we have $d_{T}(u, [w,v])\ge 2$ and $d_{T}(v, [u,w])\ge 2$.   Axiom 2  thus follows from Lemma~\ref{lem:easy1}.  


On the other hand, suppose that $d(w,[u,v])=1$. Let $o\in [u,v]$ be adjacent to $w$ and consider the vertices $u',v'\in \Lk(o)\cap [u,v]$ which lie in $[u,o]$ and $[o,v]$, respectively.
If $u\neq u'$, then $d_{L_u}(L_w,L_v)=0$, and so we may assume without loss of generality that $u=u'$. Furthermore,  $\pi_{L_u}(L_v)=\pi_{L_u}(L_{v'})$ by definition. Thus to prove the upper bound on $d_{L_u}(L_w,L_v)$, it suffices to assume that $v=v'$, in which case $u,v,w$ all lie in $\Lk(o)$, where $o\in \mathcal V_2$.

Let $e=[w,o]$,  $e_1=[u,o],$ and $ e_2=[v,o]$.  It follows from Lemma~\ref{lem:easy3} that 
\[
\frac{1}{\lambda} d_{\ell_{e}}(\ell_{e_1}, \ell_{e_2}) -\lambda \le d_{L_{w}} (L_{u}, L_{v})     \le \lambda d_{\ell_{e}}(\ell_{e_1}, \ell_{e_2}) + \lambda.
\] 
 Again applying Lemma~\ref{lem:easy3} with the roles of $u, v, w$ exchanged, we have that 
\[
\frac{1}{\lambda} d_{\ell_{e_1}}(\ell_{e}, \ell_{e_2}) - \lambda \le d_{L_{u}} (L_{w}, L_{v} ) \le \lambda d_{\ell_{e_1}}(\ell_{e}, \ell_{e_2}) + \lambda.
\]
Since $
\{ \ell_{f}   \, |\, f \in E(T) \textrm{ and } f^{+} = o \}
$
satisfies the projection axioms with constant $\xi_0$, it follows that  $d_{\ell_e}(\ell_{e_1}, \ell_{e_2})>\xi_0$ implies that $d_{\ell_{e_1}}(\ell_{e}, \ell_{e_2})\le \xi_0$.  Since there are finitely many choices for $o$ up to the action $G'$, the constant $\xi_0$ may be chosen independently of $o$.
Thus, setting $ \xi = \lambda \xi_{0} + \lambda$,  the above inequalities show that  $d_{L_W}(L_u,L_v)>\xi$ implies that $d_{L_u}(L_w,L_v)\leq \xi$.  This verifies Axiom 2.

\textbf{Axiom 3:}
 For distinct $u, v \in \mathcal V_1$, we will prove  the set 
\[
\{w\in \mathcal V_1 \mid d_{L_w}(L_u, L_v )>\xi\}
\]
is finite. By Lemma~\ref{lem:easy1}, any  such vertex $w$ is either contained in the interior of $[u,v]$ or satisfies $d(w, [u,v])=1$.
The first case yields at most $d(u, v)-1$ choices for $w$.

Suppose $d(w, [u,v])=1$. As in the proof of Axiom~2,   we can assume  that    $u,v,w$ lie in  $\Lk(o)$  for some vertex $o$ in $\mathcal{V}_{2}$.  
Let  $e_1=[u,o], e_2=[v,o]$, and $e = [w,o]$.
By Lemma~\ref{lem:easy3}, we have $
d_{L_{w}} (L_{u}, L_{v}) \le \lambda d_{\ell_{e}}(\ell_{e_1}, \ell_{e_2}) + \lambda$.
Since $\xi =  \lambda \xi_{0} + \lambda$, it follows that 
\[
\{w \in \Lk(o) \, \bigl | \,d_{L_w}(L_u, L_v ) > \xi \} \subset \{w \in \Lk(o) \, \bigl |\, d_{\ell_e}(\ell_{e_1}, \ell_{e_2}) > \xi_0 \}.
\]

The projection axioms for $
\{ \ell_{f} \, |\, f \in E(T) \textrm{ and } f^{+} = v \}
$ imply  that the latter set is finite, and so the former set must also be finite. 
Since there are finitely many possibilities for $o$, this verifies Axiom 3.
\end{proof}

\begin{lem}
For each $i =1,2$, the action of $G = \pi_1(\mathcal{G})$ on the collection $\mathbb L_i= \{ L_v \mid v\in\mathcal V_i\}$ satisfies
\[
\Pi_{g L_v}(g L_u) = g \Pi_{L_v}( L_u)
\]
for any $v\in \mathcal V_i$ and any $g\in G$.
\end{lem}

\begin{proof}
This follows immediately from the definition of $\Pi$ and the fact that the maps $\proj$ and $\psi$ are $G$--equivariant in the sense that $\proj_{\overline{ge}}(\ell_{gf})=g\cdot \proj_{\bar e}(\ell_f)$ and $\psi_{ge}(gx)=g\psi_e(x)$.
\end{proof}

We are now ready to prove Theorem~\ref{thm1}.

\begin{proof}[Proof of Theorem~\ref{thm1}]
Let $G'$ be the finite index subgroup of $G$ given by Lemma~\ref{lem:index2subgroup}, which is also a Croke-Kleiner admissible group. Without loss of generality, we replace $G$ by $G'$ for the rest of the proof.

By Proposition~\ref{prop:Fi},  the collection of  quasi-lines $\mathbb L_i= \{ L_v \mid v\in\mathcal V_i\}$  satisfies the projection axioms with a constant $\xi$ for $i=1,2$.  Fix $K>4\xi$. The unbounded quasi-trees of metric spaces $\mathcal C_{K} (\mathbb L_1)$ and $\mathcal C_{K} (\mathbb L_2)$ are themselves quasi-trees, and they admit 
unbounded isometric actions  $G \curvearrowright \mathcal C_{K} (\mathbb L_1)$ and $G \curvearrowright \mathcal C_{K}(\mathbb L_2)$.

Since the underlying graph of $G$ is bipartite, we can choose an edge $\check{e}$ in $\Gamma$ which is not a loop. Choosing the orientation of $\check{e}$ correctly, $\mu = \check{e}^{-}$ and $\omega = \check{e}^{+}$ have lifts in $T$ belonging to $\mathcal{V}_1$ and $\mathcal{V}_2$ respectively.
By  construction,  elements of $Z_{\mu}$ and $Z_{\omega}$ are loxodromic and elliptic in the action on $\mathcal{C}_{K}(\mathbb{L}_1)$, respectively, and elliptic and loxodromic in the action on $\mathcal{C}_{K}(\mathbb{L}_2)$, respectively. By Lemma~\ref{lem:key}, we conclude that the group $G$ is $\mathcal{H}$--inaccessible. 
\end{proof}

Every graph manifold has a finite cover that is a graph manifold $N$ containing at least two Seifert fibered spaces such that each Seifert fibered piece has orientable, hyperbolic base orbifold. We call such a graph manifold  {\it non-elementary} in Section~\ref{sec:NPC}. Since $\pi_1(N)$ is a Croke-Kleiner admissible group, the following corollary is immediate from Theorem~\ref{thm1}.

\begin{cor}
\label{cor:graphmanifoldfi}
Every graph manifold has a finite cover whose fundamental group is $\mathcal{H}$--inaccessible.
\end{cor}

It is still unknown whether $\mathcal{H}$--inaccessibility of a finite index normal subgroup of $G$ passes to $\mathcal{H}$--inaccessibility of the ambient group $G$. Thus it is  natural to ask whether the ``finite cover'' condition  in Corollary~\ref{cor:graphmanifoldfi} can be removed. We will address this question in the following section.

\section{$\mathcal{H}$--accessibility of 3-manifold groups}
\label{sec:NPC}

The goal of this section is to prove Theorem~\ref{thm:NPCmld}, which gives conditions under which the fundamental group of a non-geometric 3--manifold is $\mc H$--inaccessible.

We begin by recalling some definitions and facts about 3-manifolds.  Let $M$ be a compact, connected, orientable, irreducible 3-manifold with empty or toroidal boundary.  By the geometrization theorem for 3-manifolds of Perelman (\cite{Perel1}, \cite{Perel2}, \cite{Perel3}) and Thurston, either
\begin{enumerate}
    \item the manifold $M$ is \emph{geometric}, in the sense that its interior admits one of the following geometries: $S^3$, $\mathbb{E}^3$, $\mathbb{H}^3$, $S^{2} \times \mathbb{R}$, $\mathbb{H}^{2} \times \mathbb{R}$, $\widetilde{SL(2, \mathbb{R})}$, Nil, and Sol; or 

    \item the manifold $M$ is {\it non-geometric}. In this case, the  torus decomposition  of 3-manifolds yields a nonempty minimal union $\mathcal{T} \subset M$ of disjoint essential tori, unique up to isotopy, such that each component of $M \backslash \mathcal{T}$ is either a Seifert fibered piece or a hyperbolic piece. 
\end{enumerate}

We refer the reader to \cite{Scott} for background on geometric structures on 3-manifolds. A Seifert fibered piece is called \emph{non-elementary} if its base orbifold is orientable and hyperbolic, and it is called \emph{isolated} if it is not glued to any other Seifert fibered piece.

The manifold $M$ is called a \emph{graph manifold} if all the pieces of $M \backslash \mathcal{T}$ are Seifert fibered. A graph manifold is \emph{non-elementary} if it contains at least two pieces and all pieces are non-elementary. In other words, a non-elementary graph manifold is obtained by gluing at least two and at most finitely many non-elementary Seifert fibered manifolds, where the gluing maps between the Seifert components do not identify (unoriented)
Seifert fibers up to homotopy.

We will call a non-geometric manifold $M$ a \emph{mixed manifold} if it is not a graph manifold.  If there is a subcollection $\mc T'$ of $\mc T$ and a connected component of $M \backslash \mc T'$ that is a graph manifold, then this connected component is called a \emph{graph manifold component} of the mixed manifold $M$.  A graph manifold component is \textit{maximal} if it is not properly contained in another graph manifold component.  A mixed manifold is \textit{non-elementary} if all maximal graph manifold components and Seifert fibered pieces are non-elementary.

\begin{rem}
 Every graph (respectively, mixed) manifold  is finitely covered by a non-elementary graph (respectively, mixed) manifold (see, for example,  \cite[Lemma~3.1]{PW14}, \cite[Lemma~2.1]{KL98}).   
\end{rem}

Our starting point for proving Theorem~\ref{thm:NPCmld} is the following lemma, which describes when  $\pi_1(M)$ is relatively hyperbolic.
 
\begin{lem}[{\cite{Dah03, BW13}}]\label{lem:RHmfld}
Let $M_1, \dots, M_k$ be the maximal graph manifold components and Seifert fibered pieces of the torus decomposition of $M$. Let $S_1, \dots, S_{\ell}$ be the tori in the boundary of $M$ that bound a hyperbolic piece, and let $T_1, \dots,T_m$ be the tori in the torus decomposition of $M$ that separate two hyperbolic components. Then $\pi_1(M)$ is hyperbolic relative to 
\[
\mathbb{P} = \{\pi_1(M_p)\}_{p=1}^k \cup \{\pi_1(S_q)\}_{q=1}^\ell \cup \{\pi_1(T_r)\}_{r=1}^m.
\]
\end{lem}

This relatively hyperbolic structure on $\pi_1(M)$ is useful because of the following result, which gives a criterion for relatively hyperbolic groups to be $\mc H$--inaccessible.

\begin{lem}
\label{lem:RHG}
Let $(G, \mathbb{P})$ be a relatively hyperbolic group. If there is a peripheral subgroup $P \in \mathbb{P}$ that satisfies the hypotheses of Lemma~\ref{lem:key}, then $G$ is $\mathcal{H}$--inaccessible.
\end{lem}

Before proving the lemma, we state an immediate corollary, which gives a different proof of \cite[Theorem~6.2]{AR19}.

\begin{cor}[{\cite[Theorem~6.2]{AR19}}]
\label{cor:hyp}
The fundamental group of a finite-volume cusped hyperbolic 3-manifold is  $\mathcal{H}$--inaccessible.
\end{cor}

We now turn to the proof of Lemma~\ref{lem:RHG}.

\begin{proof}[Proof of Lemma~\ref{lem:RHG}]
To see that $\mathcal{H}(G)$ does not contain a largest element, we will construct two actions of $G$ on hyperbolic spaces  with commuting elements $a, b \in G$ that satisfy the hypotheses of Lemma~\ref{lem:key}.  To do this, we will apply the machinery of \emph{induced actions} from \cite{AHO}.  

Since $P$ satisfies the hypotheses of Lemma~\ref{lem:key} by assumption, there are commuting elements $a,b\in P$ and isometric actions $P \curvearrowright X$ and $P \curvearrowright Y$ on  hyperbolic spaces  such that $a$ and $b$  act loxodromically and elliptically, respectively, in the action $P \curvearrowright X$, and $b$ acts loxodromically  in the action $P \curvearrowright Y$.  For all $Q\in \mathbb P\setminus \{P\}$, fix the trivial action of $Q$ on a point.  By \cite[Corollary~4.11(a)]{AHO}, there exist hyperbolic spaces $Z_X,Z_Y$ on which $G$ acts by isometries,  associated to the collection of actions $\{Q\curvearrowright * \mid Q\in \mathbb P\setminus \{P\}\}\cup \{P\curvearrowright X\}$ and the collection of actions $\{Q\curvearrowright * \mid Q\in \mathbb P\setminus \{P\}\}\cup \{P\curvearrowright Y\}$, respectively.  Moreover, there are coarsely $P$--equivariant quasi-isometric embeddings $X\to Z_X$ and $Y\to Z_Y$.  Therefore, $a$ acts loxodromically and $b$ acts elliptically in the action $G\curvearrowright Z_X$, while $b$ acts loxodromically in the action $G\curvearrowright Z_Y$. This completes the proof.
\end{proof}

In light of Lemmas~\ref{lem:RHmfld} and \ref{lem:RHG}, to prove the $\mc H$--inaccessibility of $\pi_1(M)$, it suffices to understand its peripheral subgroups. In Section~\ref{subsec:SF}, we analyze the fundamental groups of the Seifert fibered pieces.  The more difficult subgroups to understand are the fundamental groups of the maximal graph manifold components. We consider these in Section~\ref{sec:NPCgraph} and give conditions under which they satisfy Lemma~\ref{lem:key}; see Proposition \ref{prop:NPCgraph}.  In Section~\ref{sec:mixedmanifolds}, we put these results together and prove Theorem~\ref{thm:NPCmld}.  Up to this point, we have been assuming that  $M$ has empty or toroidal boundary. Finally, in Section~\ref{sec:Boundary}, we consider $3$-manifolds with higher genus boundary components.

\subsection{Seifert fibered manifolds}
\label{subsec:SF}

In this section, we  analyze Seifert fibered pieces.

\begin{lem}
\label{lem:centralextension}
Let $1 \to \mathbb{Z} \xrightarrow{i} G \xrightarrow{\pi} H \to 1$ be a short exact sequence where $\Z$ is central in $G$ and $H$ is a non-elementary hyperbolic group. Then $G$ is $\mathcal{H}$--inaccessible.
\end{lem}
\begin{proof}
Choose a finite generating set $J$ of $H$ and consider the hyperbolic action $G \curvearrowright \Cay(H, J)$. Let $a$ be a generator of the group $\Z$, and let $b$ be an element of $G$ such that $\pi(b)$ is loxodromic in $H \curvearrowright \Cay(H, J)$. The element $b$ is thus loxodromic in the action $G \curvearrowright \Cay(H, J)$, as well, while the element $a$ is elliptic (in fact, trivial) in this  action.

    Since every integral cohomology class of a hyperbolic group is bounded (see \cite{Min01}), the central extension $\Z \to G \to H$ corresponds to a bounded element of $H^{2}(H, \Z)$. Hence \cite[Lemma~4.1]{HRSS22} provides a quasi-morphism $\phi \colon G \to \Z$ which is unbounded on $i(\Z)$. By \cite[Lemma~4.15]{ABO19}, there exists a generating set $S$ for $G$ such that $L\vcentcolon = \Cay(G, S)$ is a quasi-line and the inclusion $\Z \to L$ induced by $i$ is a $\Z$--equivariant quasi-isometry. We thus obtain a hyperbolic action $G \curvearrowright L$ for which $a$ is loxodromic. Since $a \in Z(G)$, the elements $a$ and $b$ commute. By Lemma~\ref{lem:key}, $G$ is $\mathcal{H}$--inaccessible.
\end{proof}

\begin{cor}
\label{cor:SF}
Let $M$ be a non-elementary Seifert fibered manifold. Then $\pi_1(M)$ is $\mathcal{H}$--inaccessible.
\end{cor}

\begin{proof}
 Let $\varphi \colon M \to \Sigma$ be a Seifert fibration.
Since  $S^1 \to M \to \Sigma$ is a circle bundle over $\Sigma$, there is a short exact sequence
\[
1 \to \Z \to \pi_1(M) \to \pi_1(\Sigma) \to 1,
\] where $\Z$ is the normal cyclic subgroup of $\pi_1(M)$ generated by a fiber. The group $\Z$ is central in $\pi_1(M)$ since $\Sigma$ is orientable (see, e.g., \cite[Proposition~10.4.4]{Mar}). 
By Lemma~\ref{lem:centralextension}, the group $\pi_1(M)$ is $\mathcal{H}$--inaccessible.  
\end{proof}

\subsection{$\mathcal{H}$--accessibility of non-elementary graph manifolds}\label{sec:NPCgraph}
Let $M$ be a 3-dimensional non-elementary graph manifold with Seifert fibered pieces $M_1,\ldots,M_k$ in its torus decomposition. There is an induced graph-of-groups structure $\mathcal{G}$ on $\pi_1(M)$ with underlying graph $\Gamma$ as follows. There is a vertex of $\Gamma$ for each $M_i$, with vertex group $\pi_1(M_{i})$. Each edge group is $\Z^2$, the fundamental group of a torus in the decomposition. The edge monomorphisms come from the two different gluings of the torus into the two adjacent Seifert fibered components. With this graph of groups structure, $\pi_1(M)$ is a Croke-Kleiner admissible group.

The universal cover $\widetilde{M}$ of $M$ is tiled by a countable collection of copies of the
universal covers $\widetilde{M}_1, \dots, \widetilde{M}_k$. We call these subsets {\it vertex spaces}.
We refer to boundary components of vertex spaces as {\it edge spaces}. Two
vertex spaces are either disjoint or intersect along an edge space. Let $T$
be the
Bass-Serre tree of $\mathcal G$.

 Applying Theorem~\ref{thm1} to the Croke-Kleiner admissible group $\pi_1(M)$, we obtain a cover $M' \to M$ of degree $2$ such that $\pi_1(M')$ is not $\mathcal{H}$--accessible.  
However, this is not enough to conclude $\mathcal{H}$--inaccessibility of $\pi_1(M)$, as it is unknown whether $\mathcal{H}$--inaccessibility of a finite index subgroup passes to the ambient group.
In this section, we will show that $\pi_1(M)$ itself is  $\mathcal{H}$--inaccessible; see Proposition~\ref{prop:NPCgraph}.

We begin with a lemma.  Let $\pi_{\alpha}(
\beta)$ be the closest point projection of a line $\beta$ to a line $\alpha$ in a hyperbolic space. Let $d_\alpha(\cdot, \cdot)$ be in the resulting projection distances.

\begin{lem}
\label{lem:important}
Let $F$ be a 2-dimensional hyperbolic orbifold with nonempty boundary and universal cover $\widetilde{F}$, and let $\mathbb L$ be the collection of boundary lines of $\widetilde{F}$.   For any   $\alpha\in \mathbb L$ and any loxodromic $\gamma\in \pi_1(F)$ whose axis in $\widetilde{F}$ is also a line in $\mathbb L$, the following holds.  There exists a constant $\lambda >0$ such that for any $n \in  \mathbb{Z}$ and any line $\beta \in \mathbb{L}\setminus \{\alpha, \gamma^{n}(\alpha) \}$,   we have
$$
d_{\beta}(\alpha, \gamma^{n}(\alpha)) \le \lambda.
$$
\end{lem}

\noindent The proof of this lemma is very similar to that of \cite[Lemma 5.6]{AR19}. We refer the reader to that paper for some figures that may be helpful; see in particular \cite[Figure 8]{AR19}.

\begin{proof}[Proof of Lemma \ref{lem:important}]
 Since $\mathbb L$ is a $\pi_1(F)$--invariant collection of axes in the hyperbolic plane $\mathbb H^2$, it follows from Example \ref{ex:linesinhypplane} that $(\mathbb L,\pi_\ell)$ satisfies the projection axioms for some constant $\xi$.  In particular, 
there exists a constant $\xi > 1$ such that $\operatorname{diam}(\pi_\ell(\ell')) \le \xi$ for distinct elements $\ell$ and $\ell'$ in $\mathbb{L}$. Let 
$$
\lambda = \max \{ \xi, d(\alpha, \gamma(\alpha)) + 2\xi, d(\alpha,\gamma^{2}(\alpha)) + 2\xi \}.
$$

Let $\ell\in \mathbb L$ denote the axis of $\gamma$ in $\widetilde{F}$.   
If $\ell=\alpha$, then 
\[d_\beta(\alpha,\gamma^n(\alpha))=d_\beta(\alpha,\alpha)\leq \xi\leq \lambda,\]
and the result holds.  
For the remainder of the proof, we assume $\ell\neq\alpha$ and consider two cases.

   \noindent \emph{Case~1}: $\beta \notin \{\gamma^{k}(\alpha)  \,| k \in \mathbb Z \}$
    
In this case,  there exists a unique $k_0 \in \mathbb Z$ such that $\beta$ lies between $\gamma^{k_{0}}(\alpha)$ and $\gamma^{k_{0}+1}(\alpha)$. That is, $\partial \mathbb H^2$ minus the endpoints of $\gamma^{k_0}(\alpha)$ and $\gamma^{{k_0}+1}(\alpha)$ consists of four intervals, one containing the endpoints of $\beta$, one containing the endpoints of all $\gamma^i(\alpha)$ for $i\notin \{k_0,k_0+1\}$, and the other two disjoint from all endpoints of lines in $\mathbb L$.
Fixing an appropriate orientation on $\beta$, we partially order sub-intervals $I=[x,y]$  and $J=[z,w]$ of $\beta$ (with $x\leq y$ and $z\leq w$ in the orientation) by $I \leq J$ if $x\leq z$ and $y\leq w$. Then
the projections of the lines $\gamma^{k}(\alpha)$ onto $\beta$ occur in the following order 
$$
\pi_{\beta}(\gamma^{k_{0}}(\alpha)) < \pi_{\beta}(\gamma^{k_{0}-1}(\alpha))  < \pi_{\beta}(\gamma^{k_{0} -2}(\alpha))  < \ldots < \pi_{\beta}(\ell)
$$ and
$$
\pi_{\beta}(\ell) < \ldots < \pi_{\beta}(\gamma^{k_{0} +3}(\alpha)) < \pi_{\beta}(\gamma^{k_{0} + 2}(\alpha)) < \pi_{\beta}(\gamma^{k_{0} +1}(\alpha)).
$$
Thus
$$
d_{\beta}(\alpha, \gamma^{n}(\alpha)) \le d_{\beta} (\gamma^{k_{0}}(\alpha), \gamma^{k_{0} +1}(\alpha)) \le  d(\gamma^{k_{0}}(\alpha), \gamma^{k_{0} +1}(\alpha)) + 2\xi,
$$
where $d(\gamma^{k_{0}}(\alpha), \gamma^{k_{0} +1}(\alpha))$ denotes the distance between $\gamma^{k_{0}}(\alpha)$ and $\gamma^{k_{0} +1}(\alpha)$ in the hyperbolic plane.  The final inequality follows from the fact that the nearest point projection is a $1$--Lipschitz map and that $\pi_\ell(\ell')$ has diameter at most $\xi$ for any distinct lines $\ell,\ell'\in \mathbb L$.
Since $d(\gamma^{k_{0}}(\alpha), \gamma^{k_{0} +1}(\alpha))) = d(\alpha, \gamma(\alpha))$, it follows that $$d_{\beta}(\alpha, \gamma^{n}(\alpha)) \le  d(\alpha, \gamma(\alpha)) + 2\xi \le \lambda.$$

\noindent\emph{Case~2}: $
\beta = \gamma^{k}(\alpha)$ for some integer $k \neq 0, n$.
Using an analogous argument to Case 1, we see that $d_{\beta}(\alpha, \gamma^{n}(\alpha))$ is bounded above by \[d_{\beta}(\gamma^{k-1}(\alpha), \gamma^{k+1}(\alpha)) \le d(\alpha, \gamma^{2}(\alpha)) + 2\xi \le \lambda. \qedhere\]
\end{proof}

\begin{prop}
\label{prop:NPCgraph}
The fundamental group of a non-elementary graph manifold  is  $\mathcal{H}$--inaccessible.
\end{prop}

\begin{proof}
Let $\mathcal{G}$ be the graph-of-groups structure on $\pi_1(M)$ with  underlying graph $\Gamma$ described at the beginning of this section.
The assumption that the graph manifold $M$ is non-elementary ensures that there are at least two vertices in the graph $\Gamma$. We divide the proof into two cases, depending on the location of loops in $\Gamma$.

 Fix an edge $\alpha$ in $\Gamma$ that is not a loop, and label the vertex $\alpha^-$ by $\mu$ and the vertex $\alpha^+$ by $\omega$.  Let $T_{\alpha}$ be the  torus in $M$ associated to the edge $\alpha$. Let $v$ and $w$ be two adjacent vertices in the tree $T$ such that $\widetilde{M}_v$ and $\widetilde{M}_w$ are the universal covers of the Seifert pieces $M_{\mu}$ and $M_{\omega}$, respectively. Let $z_\mu$ and $z_\omega$ be the generators of $Z_\mu$ and $Z_\omega$, respectively.  

\noindent {\it Case~1:} Suppose there is no loop in $\Gamma$ based at the vertex $\mu$.  Let 
\[
\mathbb{W}_{\mu} \vcentcolon= \bigl \{ L_v \, \bigl |\, \text{$\widetilde{M}_v$ is a lift of the Seifert fibered piece $M_\mu$} \bigr \}
\]

If $L_v$ and $L_{v'}$ are two distinct elements in $\mathbb{W}_\mu$, then $d(v, v') \ge 2$ (though they are not necessarily an even distance apart).   In this case, the techniques in Section~\ref{CKA:finiteindex} apply  to show that there is a cobounded action $\pi_1(M) \curvearrowright \mathcal{C}_{K}(\mathbb{W}_\mu)$ such that $z_\mu$ is loxodromic and $z_\omega$ is elliptic. 

\noindent {\it Case~2:}
Suppose there is a loop in $\Gamma$ based at the vertex $\mu$.

As in Section~\ref{CKA:finiteindex}, we partition the vertex set $T^0$ into two disjoint collections of vertices $\mathcal{V}_1$ and $\mathcal{V}_2$ such that if $z$ and $z'$ both lie in $\mathcal{V}_i$ then $d(z, z')$ is even.
 Applying Theorem~\ref{thm1} to the Croke-Kleiner admissible group $\pi_1(M)$, we obtain a degree 2 cover $M' \to M$ such that $\pi_1(M')$ is $\mathcal{H}$--inaccessible and $\pi_1(M')$ preserves $\mc V_1$ and $\mc V_2$.

Assume without loss of generality that $v$ is in $\mathcal{V}_1$ and $w$ is in $\mathcal{V}_2$. Let
\[
\mathbb{Q}_{\mu} \vcentcolon= \bigl \{ L_u \, \bigl | \, \text{$u \in \mathcal{V}_1$ and $\widetilde{M}_{u}$ is a lift of $M_{\mu}$} \bigr. \}
\] 
The results in Section~\ref{CKA:finiteindex} still hold for $\mathbb{Q}_\mu$, and thus we obtain quasi-trees of spaces $\mathcal{C}_{K} (\mathbb{Q}_\mu)$ for  sufficiently large $K$. Since $\pi_1(M')$  preserves $\mathbb{Q}_\mu$, we obtain an action $\pi_1(M') \curvearrowright \mathcal{C}_{K} (\mathbb{Q}_\mu)$  as in Section~\ref{CKA:finiteindex}.

Passing to a power of two if necessary, we may assume that $z_{\mu}, z_\omega \in \pi_1(M')$.
As shown in the proof of Theorem~\ref{thm1}, the element $
z_{\mu}$ acts loxodromically on $\mathcal{C}_{K}(\mathbb{Q}_\mu)$, while 
$z_{\omega}$   acts elliptically on $\mathcal{C}_{K}(\mathbb{Q}_\mu)$.
 By the construction of $\mc C_K(\mathbb Q_\mu)$, and since vertex groups are central extensions of $\Z$, the element $z_{\mu}$ is a WWPD${}^+$ element in the action $\pi_1(M')\curvearrowright \mc C_K(\mathbb Q_\mu)$.
Hence Proposition~\ref{prop1} provides a homogeneous quasimorphism $q_K \colon \pi_1(M') \to \mathbb R$ satisfying $q_K(z_{\omega}) =0$ and $q_K(z_{\mu}) \neq 0$.  

Our goal is to extend $q_K$ to a homogeneous quasimorphism $\pi_1(M)\to\mathbb R$ while ensuring that $z_{\omega}$ and $z_{\mu}$ still have trivial and non-trivial image, respectively. 
Let $h \in \pi_1(M)$  be a representative of the non-trivial coset of $ \pi_1(M')$ in $ \pi_1(M)$. Define a function $q'_{K} \colon \pi_1(M') \to \mathbb{R}$  by  
\[q'_{K}(x) \vcentcolon= q_{K}(x) + q_{K}(hxh^{-1}).\]  

Note that $q'_{K}$ is constant on conjugacy classes of $\pi_1(M)$, i.e., $q'_{K}(yxy^{-1}) = q'_{K}(x)$ for any $y \in \pi_1(M)$ and $x \in \pi_1(M')$). Hence it follows from the proof of \cite[Lemma~7.2]{BF09} that $q'_{K}$ extends to a homogeneous quasimorphism $\rho_K \colon \pi_1(M) \to \mathbb{R}$ defined by $\rho_K(x) \vcentcolon= q'_{K}(x^2) \bigl / 2$ for each $x \in \pi_1(M)$.

\begin{lem}
\label{lem:inproof}
Suppose there is a loop in $\Gamma$ based at $\mu$.  For $K$ large enough, we have $\rho_K(z_{\omega}) =0 $ and $\rho_K(z_{\mu}) \neq 0$.
\end{lem}

We defer the proof of the  lemma for the moment and assume this result to complete the proof Proposition~\ref{prop:NPCgraph}. 
Since $\rho_K \colon \pi_1(M) \to \mathbb{R}$ is a nonzero homogeneuous quasimorphism, we obtain from Proposition~\ref{prop2} an action $\pi_1(M) \curvearrowright \mathcal{L}$ on a quasi-line.  Moreover, since $\rho_{K}(z_{\mu}) \neq 0$ and $\rho_K(z_{\omega})=0$, the element $z_{\mu}$ is loxodromic while $z_{\omega}$ is elliptic in this action.

Now, consider the other endpoint $\omega$ of $\alpha$.  Suppose first there is not a loop in $\Gamma$ based at $\omega$.  Interchanging the roles of $\mu$ and $\omega$ in Case 1 above produces an action $\pi_1(M)\curvearrowright \mathcal C_K(\mathbb W_\omega)$ such that $z_\mu$ is elliptic and $z_\omega$ is loxodromic.  On the other hand, if there is a loop in $\Gamma$ based at $\omega$, then interchanging the roles of $\mu$ and $\omega$ in Case 2 above produces an action $\pi_1(M)\curvearrowright \mathcal L'$ on a quasi-line in which (after possibly passing to a power of 2) $z_\mu$ is elliptic and $z_\omega$ is loxodromic.

Regardless of which combination of cases holds for the  vertices $\mu$ and $\omega$, we have produced two actions on hyperbolic spaces and two commuting elements $z_\mu$ and $z_\omega$ which satisfy the conditions of Lemma~\ref{lem:key}, which concludes the proof.
\end{proof}

 We now prove Lemma~\ref{lem:inproof}.
 \begin{proof}[Proof of Lemma~\ref{lem:inproof}]

Recall that $v\in \mc V_1$. As $h \in \pi_1(M)$ is a representative of the non-trivial coset of $\pi_1(M')$ in $\pi_1(M)$, we have $hv \in \mathcal{V}_2$. Note that $\widetilde{M}_{hv}$ is also a lift of $M_{\mu}$, even though $hv$ is not in $\mathcal{V}_1$.
Fix a vertex $v_0$ adjacent to $hv$ such that $\widetilde{M}_{v_0}$ is a lift of $M_{\mu}$ in $\widetilde{M}$.  This ensures that $L_{v_0}$ is in $\mathbb{Q}_1$. Let $l \in 
\mathbb{L}_{hv}$ be the boundary line of $\widetilde{F}_{hv}$ corresponding to the edge $[v_0, hv].$

We will first show that $\rho_K(z_{\omega}) = 0$. Since $q_K$ is a homogeneous quasimorphism and $z_\omega\in \pi_1(M')$, we have that 
\[
\rho_{K}(z_{\omega}) = q'_{K}(z_{\omega}) = q_{K}(z_{\omega}) + q_{K}(hz_{\omega}h^{-1}) = 0 + q_{K}(hz_{\omega}h^{-1}).
\]
By Proposition~\ref{prop1}, to show  $\rho_K(z_{\omega}) = q_{K}(hz_{\omega}h^{-1}) =  0 $, it suffices to show that $hz_{\omega}h^{-1}$ is elliptic in the action $\pi_1(M') \curvearrowright \mathcal{C}_{K}(\mathbb{Q}_1)$.   
 Let $\xi >0$ be the projection constant of the projection complexes $\mathbb{Q}_1$ and $\mathbb Q_2$. 
  Since $M_{hv}$ is a Seifert fibered piece, we have $\widetilde{M}_{hv}= \widetilde{F}_{hv} \times \mathbb{R}$, where $F_{hv}$ is the base orbifold of $M_{hv}$.  Applying Lemma~\ref{lem:important} to the space ${F}_{hv} $, the collection of boundary lines of $\widetilde{F}_{hv}$, the fixed boundary line $l$, and the chosen element $\gamma=h z_\omega h^{-1}$, we obtain a constant $\lambda >0$. We further enlarge $\lambda$ so that it satisfies Lemma~\ref{lem:easy3}.
 
 Choose $K > 4\xi + 4 +  2\lambda  +  \lambda^2 $ large enough to apply Proposition \ref{prop:BBFDistanceProp}, and let $y_0$ be a point in the projection of $L_{hz_\omega h^{-1} v_0}$ to $L_{v_0}$. 
We will show that  $ d_{\mathcal C_K(\mathbb Q_1)}(y_0, \gamma^{n} (y_0)) \le 6K$ for all $n \in \mathbb{Z}$, which will imply that $\gamma$ is elliptic in the action $\pi_1(M') \curvearrowright \mathcal{C}_{K}(\mathbb{Q}_1)$, as desired.

Fix $n\in \mathbb Z$.  By Proposition~\ref{prop:BBFDistanceProp}, we have
\begin{equation}
    \label{equ:4}
    d_{\mathcal C_K(\mathbb Q_1)}(y_0, \gamma^{n}( y_0)) \le 4 \sum_{\substack{u \in \mathcal{V}_1 \\ \text{$L_u \in \mathbb{Q}_1$}}} [d_{L_u}(y_0,  \gamma^{n}(y_0))]_K + 6K.
\end{equation} 
Thus it suffices to show that $d_{L_u}(y_0,  \gamma^{n}(y_0))< K$ for all $u\in \mathcal V_1$ such that  $L_u\in \mathbb Q_1$.  Since $L_u\in \mathbb Q_1$, $\widetilde M_u$ is a lift of $M_\mu$.

We divide the proof into  several cases, depending on the location of the vertex $u$. 

\noindent \textit{Case 1: $u\in\{v_0,\gamma^n(v_0)\}$.} 
We  assume that $u = v_0$ as the case $u=  \gamma^{n}(v_0)$ is proved similarly. 

By assumption, $y_0\in L_{v_0}$, and so  $\gamma^{n}(y_0) \in L_{\gamma^{n}(y_0)}$. By definition,  
$$ d_{L_{v_0}}(y_0, \gamma^{n}(y_0)) = \diam (\{y_{0}\} \cup \Pi_{L_{v_0}}(L_{\gamma^{n}(v_0)}))$$ and 
$$d_{L_{v_0}} (L_{\gamma(v_0)}, L_{\gamma^{n}(v_0)}) = \diam \bigl (\Pi_{L_{v_0}}(L_{\gamma(v_0)}) \cup \Pi_{L_{v_0}}(L_{\gamma^{n}(v_0)}) \bigr ).$$
As $y_0 \in \Pi_{L_{v_0}}(L_{\gamma(v_0)})$ and the diameter of $\Pi_{L_{v_0}}(L_{\gamma(v_0)})$ is no more than $\xi$, it follows that 
\begin{equation}
    \label{equ:1}
    \bigl | d_{L_{v_0}}(y_0, \gamma^{n}(y_0)) - d_{L_{v_0}} (L_{\gamma(v_0)}, L_{\gamma^{n}(v_0)}) \bigr | \le 2 \xi.
\end{equation}

The line $l$ is the boundary line of $\widetilde{F}_{hv}$ associated to the edge $[v_0, hv]$. Recall that  $z_\omega$ is an element of the edge group $G_{[v,w]}$, and so it fixes the vertex $v$. 
 Thus $\gamma(hv)=hz_\omega h^{-1}(hv)=h z_\omega(v)=hv$,
 and so the lines $\gamma(l)$ and $\gamma^{n}(l)$ are
 the boundary lines in $\widetilde{F}_{hv}$ associated to the edges $[hv, \gamma(v_0)]$ and $[hv, \gamma^{n}(v_0)]$, respectively.
 
 Combining (\ref{equ:1})  with Lemmas~\ref{lem:easy3} and \ref{lem:important} implies that 
\begin{align*}
d_{L_{v_0}}(y_0,  \gamma^{n}(y_0)) & \le 
  d_{L_{v_0}}(L_{\gamma (v_0)},  L_{\gamma^{n}(v_0)}) + 2 \xi \\ 
    & \le \lambda  d_{l}(\gamma(l), \gamma^{n}(l)) + \lambda + 2\xi \\
    &= \lambda \, d_{\gamma^{-1}(l)} (l, \gamma^{n-1}(l)) + \lambda + 2\xi \le \lambda^2 + \lambda + 2\xi < K.
\end{align*}

\noindent \textit{Case 2:  $u\in \Lk(hv)$ but $u\not\in\{v_0,\gamma^n(v_0)\}$.} 
 Let $b$ be the boundary line of $\widetilde{F}_{hv}$  corresponding to the edge $[u, hv]$, so that $b \notin \{ l, \gamma^{n}(l) \}$. By Lemma~\ref{lem:important}, we have that $
d_{b}(l , \gamma^{n}(l)) \le \lambda$.
It follows from Lemma~\ref{lem:easy3} that 
 \begin{align*}
     d_{L_u}(y_0,  \gamma^{n}(y_0)) &= d_{L_u}(L_{v_0}, L_{\gamma^{n}(v_0)}) \\
     &\le \lambda \,d_{b}(l , \gamma^{n}(l)) + \lambda \\
     &\le \lambda^2 + \lambda < K.
 \end{align*}

\noindent \textit{Case 3: $u\not\in \Lk(hv)$.}
  In this case,
$d(u, [v_0,\gamma^{n}(v_0)])\ge 2$, and so  
 $$d_{L_u}(y_0,  \gamma^{n}y_0)= d_{L_u}(L_{v_0}, L_{\gamma^{n}(v_0)}) \le \lambda < K.$$

 We have shown that $d_{L_u}(y_0,  \gamma^{n}(y_0))< K$ for all $u\in \mathcal V_1$ such that  $L_u\in \mathbb Q_1$.   Therefore \eqref{equ:4} shows that $d_{\mathcal C_K(\mathbb Q_1)}(y_0, \gamma^{n} (y_0)) \leq 6K$ for all $n$.  It follows that $\gamma$ is elliptic in the action $\pi_1(M')\curvearrowright \mc C_K(\mathbb Q_1)$, and so  $q(z_\omega)=0$.

To complete the proof, we need to verify that 
\[\rho_{K}(z_{\mu}) = q_{K}'(z_{\mu}) = q_{K}(z_{\mu}) + q_{K}(hz_{\mu}h^{-1}) \neq 0.
\]
Since $hz_\mu h^{-1}$ is a central element in $G_{hv} = \Stab_{G}(hv)$, it follows from   Remark~\ref{rem:loxellipaction} that $hz_\mu h^{-1}$ acts elliptically on $L_{v_0}$, and thus also on $\mathcal{C}_{K}(\mathbb{Q}_1)$.  By Proposition~\ref{prop1}, we have $q_K(hz_\mu h^{-1})=0$.  Since 
$q_{K}(z_{\mu}) \neq 0$, it follows that $\rho_K(z_\mu)\neq 0$.
 \end{proof}

Theorem~\ref{thm:graphandSFmflds} now follows immediately
 from Corollary~\ref{cor:SF} and Proposition~\ref{prop:NPCgraph}.

\subsection{Theorem~\ref{thm:NPCmld}}
\label{sec:mixedmanifolds}
In this section, we put together the above results and prove Theorem~\ref{thm:NPCmld}, whose statement we recall for the convenience of the reader.

\NPCmfld*

\begin{proof}
Let $M_1, \dots, M_k$ be the maximal graph manifold components and isolated Seifert fibered pieces of the torus decomposition of $M$. Let $S_1, \dots, S_{\ell}$ be the tori in the boundary of $M$ that bound a hyperbolic piece, and let $T_1, \dots,T_m$ be the tori in the torus decomposition of $M$ that separate two hyperbolic components of the torus decomposition. By Lemma~\ref{lem:RHmfld}, $\pi_1(M)$ is hyperbolic relative to 
\[
\mathbb{P} = \{\pi_1(M_p)\}_{p=1}^k \cup \{\pi_1(S_q)\}_{q=1}^\ell \cup \{\pi_1(T_r)\}_{r=1}^m.
\]
In all of the cases (1)--(4), the collection $\mathbb P$ is non-empty.  

In case (1), the collection $\{S_1,\dots, S_\ell\}\neq \emptyset$, while in case (2), $\{T_1,\dots, T_m\}\neq \emptyset$.  Both of these collections consist of tori. Note that $\mathbb Z^2$ is $\mc H$--inaccessible: the projections of $\Z^2$ onto each factor yield two actions on lines to which Lemma~\ref{lem:key} applies.  Thus, if $\{\pi_1(S_q)\} \cup \{\pi_1(T_r)\}$ is nonempty, then $\pi_1(M)$ is $\mathcal{H}$--inaccessible by Lemma~\ref{lem:RHG}, proving the theorem in cases (1) and (2).

Next suppose that (3) holds, so that there is an isolated non-elementary Seifert fibered piece $M_p$. By the proof of Corollary~\ref{cor:SF} we see that $\pi_1(M_p)$ has two actions to which Lemma~\ref{lem:key} applies. By Lemma \ref{lem:RHG}, $\pi_1(M)$ is $\mathcal{H}$--inaccessible.

Finally, suppose that (4) holds, so that there is a non-elementary maximal graph manifold component $M_p$.
By the proof of Proposition~\ref{prop:NPCgraph}, there are two commuting elements $a,b\in\pi_1(M_p)$  and two actions on hyperbolic spaces (in fact, quasi-trees) $\pi_1(M_p) \curvearrowright X$ and $\pi_1(M_p) \curvearrowright Y$ such that $a$ and $ b$ are elliptic and loxodromic, respectively, in $\pi_1(M_p) \curvearrowright X$ and $a$ is loxodromic in $\pi_1(M_p) \curvearrowright Y$.  Applying Lemma~\ref{lem:RHG} to $P = \pi_1(M_p)$, we conclude that $\mathcal{H}(\pi_1(M))$ contains no largest element.
\end{proof}

\subsection{$\mathcal{H}$--accessibility of finitely generated 3-manifold groups}\label{sec:Boundary}
In this section, we explain how one might reduce the study of $\mathcal{H}$--accessibility of all finitely
generated 3–manifold groups to the case of compact, orientable, irreducible, $\partial$--irreducible 3–manifold groups. In particular, we show that for any hyperbolic 3-manifold $M$ without rank-1 cusps, if $\pi_1(M)$ is finitely generated then it is $\mathcal H$-accessible. 

Let $M$ be an orientable 3-manifold with finitely generated fundamental group.
It follows from Scott's Core
Theorem that $M$ contains a compact codimension zero submanifold whose inclusion map is a homotopy equivalence \cite{ScottCore}, and thus also an isomorphism on fundamental groups. We thus can assume our 3-manifolds are compact.

The sphere-disk decomposition provides a decomposition of a compact, orientable 3-manifold $M$ into irreducible, $\partial$--irreducible pieces $M_1, \ldots, M_k$. In particular, $\pi_1(M)$ is a free product $\pi_1(M_1) * \pi_1(M_2) * \cdots *\pi_1(M_k)$. Let $G_i\vcentcolon =  \pi_1(M_i)$. Note that $\pi_1(M)$ is hyperbolic
relative to the collection $\PP = \{ G_1, \ldots, G_k \}$.
In light of Lemma~\ref{lem:RHG}, the $\mathcal{H}$--inaccessibility of $\pi_1(M)$ follows whenever some $G_i$ satisfies the conditions of Lemma~\ref{lem:key}. Hence, it suffices to investigate the $\mathcal{H}$--accessibility of the groups $G_i$.

If $M$ has empty or toroidal boundary, then $\mathcal{H}$--accessibility of $\pi_1(M)$ is  understood, except for a few sporadic cases, by Theorem \ref{thm:NPCmld}. The following proposition addresses certain manifolds with higher genus boundary.

\begin{prop}
\label{prop:highergenus}
Let $M$ be a compact, orientable, irreducible, $\partial$--irreducible 3-manifold which has at least one boundary component of genus at least $2$. Then $\pi_1(M)$ is $\mathcal H$--inaccessible under either of the following hypotheses:
\begin{enumerate}
    \item
    \label{item:higherboundary1} $M$ has trivial torus decomposition and at least one torus boundary component; or

    \item
    \label{item:higherboundary2} $M$ has non-trivial torus decomposition. 
\end{enumerate}
On the other hand, if $M$ has trivial torus decomposition and all boundary components have genus at least $2$, then $\pi_1(M)$ is $\mathcal H$--accessible. 

\end{prop}

\begin{proof}
 As in 
  \cite[Section~6.3]{Sun20}, we can paste compact hyperbolic 3-manifolds with totally geodesic boundaries to the higher genus boundary components of $M$ to obtain a finite volume hyperbolic manifold $N$ (in case $M$ has trivial torus decomposition) or a mixed 3-manifold (in case $M$ has non-trivial torus decomposition).

    If (\ref{item:higherboundary1}) holds, then the manifold $N$ has toroidal boundary, and, by assumption, there is a boundary torus $T$ for $N$ which is also a boundary torus of $M$.

The subgroup $P\vcentcolon= \pi_1(T)\simeq \mathbb Z^2$ satisfies Lemma~\ref{lem:key} and is a peripheral subgroup in the relatively hyperbolic structure on $\pi_1(N)$.  The proof of Lemma~\ref{lem:RHG} shows that there are commuting elements $a,b\in P$ and hyperbolic actions $\pi_1(N) \curvearrowright Z_X$ and $\pi_1(N) \curvearrowright Z_Y$ such that $a$ and $b$  act loxodromically and elliptically, respectively, in the action $\pi_1(N) \curvearrowright Z_X$, and $b$ acts loxodromically  in the action $G \curvearrowright Z_Y$.  As $\pi_1(M)$ is a subgroup of $\pi_1(N)$, we obtain induced actions $\pi_1(M) \curvearrowright Z_X$ and $\pi_1(M) \curvearrowright Z_Y$. Since $a, b \in \pi_1(M)$, we see that $\pi_1(M)$ is $\mathcal{H}$--inaccessible by Lemma~\ref{lem:key}. 

 If (\ref{item:higherboundary2}) holds, then $N$ has either empty or toroidal boundary and has the following properties:
\begin{enumerate}[(i)]
    \item $M$ is a submanifold of $N$ with incompressible toroidal boundary;
    \item cutting $N$ along the tori in the torus decomposition of $M$ yields the torus decomposition of $N$; and
    \item each piece of $M$ with a boundary component of genus at least $2$ is contained in a hyperbolic piece of $N$.
\end{enumerate}
In particular, it follows from (ii) and (iii) that $N$ is a mixed 3-manifold, and hence  $\pi_1(N)$ is $\mathcal{H}$--inaccessible by Theorem~\ref{thm:NPCmld}. 

In the proof of Theorem~\ref{thm:NPCmld}, we prove the $\mc H$--inaccessibility of $\pi_1(N)$  by showing there are two commuting elements $a,b\in \pi_1(T)$ for some  torus $T$ in the torus decomposition of $N$ and  isometric actions  $\pi_1(N) \curvearrowright Z_X$ and $\pi_1(N) \curvearrowright Z_Y$ on  hyperbolic spaces, and then applying Lemma \ref{lem:key}.

By (ii),  $T$ is also a torus in the  torus decomposition of $M$. Thus the induced actions  $\pi_1(M) \curvearrowright Z_X$ and $\pi_1(M) \curvearrowright Z_Y$ satisfy the hypotheses of Lemma~\ref{lem:key}, and so $\pi_1(M)$ is $\mathcal{H}$--inaccessible.

We now turn our attention to the final statement of the theorem. 
 In this case, the manifold $N$ is closed. A finitely generated subgroup $H$ of $N$ is a {\it virtual surface fiber subgroup} if $N$ admits a finite cover $N' \to N$ such that $H$ is a subgroup of $\pi_1(N')$ and $H$ is a surface fiber subgroup of $\pi_1(N')$.
Any finitely generated subgroup $H$ of $\pi_1(N)$ is either a geometrically finite Kleinian group or
a virtual surface fiber subgroup in $\pi_1(N)$ by the Covering Theorem (see \cite{Can96}) and the Subgroup Tameness Theorem (see \cite{Agol04, CG06} or   \cite[Theorem~4.1.2]{AFW15} for a statement).  In particular, $\pi_1(M)$ is either a virtual surface fiber subgroup, in which case it is hyperbolic, or 
it is geometrically finite in $\pi_1(N)$.  In the latter case, $\pi_1(M)$ is undistorted in $\pi_1(N)$ \cite[Corollary~1.6]{Hru10}, and we again conclude that $\pi_1(M)$ is hyperbolic, since undistorted subgroups of hyperbolic groups are hyperbolic. As a result, in either case, $\pi_1(M)$ is $\mathcal{H}$--accessible.  
\end{proof}


\end{document}